\documentclass[11pt]{article}
\usepackage{amscd}
\usepackage{graphics}
\usepackage{color}
\usepackage{cancel}
\usepackage{stmaryrd}
\usepackage{tikz}
\usetikzlibrary{matrix,arrows}
\usepackage{amsmath}
\usepackage{amsthm}
\usepackage{amssymb}
\usepackage{proof}
\usepackage{verbatim,cmll}
\usepackage{setspace}
\usepackage{lscape}
\usepackage{latexsym,relsize}
\usepackage{amscd}
\usepackage{multicol}
\usepackage{bussproofs}
\usepackage{txfonts}
\usepackage{tikz}
\usepackage[position=top]{subfig}
\usepackage{leqno}
\usepackage{authblk}
\usepackage{marvosym}
\DeclareSymbolFont{extraup}{U}{zavm}{m}{n}
\DeclareMathSymbol{\varheart}{\mathalpha}{extraup}{86}
\DeclareMathSymbol{\vardiamond}{\mathalpha}{extraup}{87}
\DeclareMathSymbol{\vardiamond}{\mathalpha}{extraup}{87}

\newcommand{\commment}[1]{}

\renewcommand{\phi}{\varphi}
\renewcommand{\emptyset}{\varnothing}

\renewcommand{\epsilon}{\varepsilon}

\newcommand{\nomi}{\mathbf{i}}
\newcommand{\nomj}{\mathbf{j}}

\newcommand{\cnomm}{\mathbf{m}}
\newcommand{\cnomn}{\mathbf{n}}

\newcommand{\bigamp}{\mathop{\mbox{\Large \&}}}

\theoremstyle{plain}
\newtheorem{thm}{Theorem}
\newtheorem{theorem}{Theorem}[section]

\newtheorem{example}[theorem]{Example}

\newtheorem{proposition}[thm]{Proposition}

\newtheorem{lemma}[thm]{Lemma}

\theoremstyle{definition}

\newtheorem{definition}[thm]{Definition}

\title{Correspondence Theory for Generalized Modal Algebras}
\author{Zhiguang Zhao}
\date{}
\begin{document}
\maketitle
\begin{abstract}
\noindent
In the present paper, we give a systematic study of the correspondence theory of generalized modal algebras and generalized modal spaces, in the spirit of \cite{Ce05,Ce22}. The special feature of the present paper is that in the proof of the (right-handed) topological Ackermann lemma, the admissible valuations are not the clopen valuations anymore, but values in the set $\mathcal{D}_{\mathcal{K}}(X)$ which are only closed and satisfy additional properties, not necessarily open. This situation is significantly different from existing settings using Stone/Priestley-like dualities, where all admissible valuations are clopen valuations.


\end{abstract}

\section{Introduction}

Generalized Boolean algebras are the $(\to,\land,\lor,\top)$-subreducts of Boolean algebras, i.e.\ its logic counterpart is the $(\neg,\bot)$-free fragment of classical propositional logic. It is first studied by Stone \cite{St35}. Our treatment of generalized Boolean algebras follow the spirit of Tarski algebras \cite{Ab67a,Ab67b,Ce05,Ce19b,Ce19a,Ce22,CeCa08}, which are the $(\to,\lor,\top)$-subreducts of Boolean algebras, therefore our treatment can be seen as Tarski algebras with a conjunction. 

In \cite{Ce05,Ce22}, some formulas defined in the language of modal Tarski algebras and subordination Tarski algebras are characterized in terms of first-order conditions on the dual modal Tarski spaces or subordination Tarski spaces. We follow their approach to study the correspondence theory for generalized modal algebras (which are the $(\to,\land,\lor,\top,\Box)$-subreducts of modal algebras) and generalized modal spaces.

The present paper aims at obtaining similar but more general results in the setting of generalized modal algebras, using algorithmic correspondence theory \cite{CoGhPa14,CoPa12}, which uses an algorithm $\mathsf{ALBA}$ to transform an input formula/inequality in the language of certain algebras into its first-order correspondent on the dual topological spaces. The special feature of the present paper is that in the proof of the right-handed topological Ackermann lemma, the admissible valuations are not the clopen valuations anymore, but values in the set $\mathcal{D}_{\mathcal{K}}(X)$ which are only closed and satisfy additional properties, not necessarily open. This situation is significantly different from the Stone/Priestley-like duality settings, where all admissible valuations are clopen valuations. The present work can also be taken as a step towards correspondence theory for logics which are not based on bounded algebras.

The paper is organized as follows: Section \ref{Sec:Prelim} gives the preliminaries on generalized modal algebras, generalized modal spaces and their duality on the object level. Section \ref{Sec:Syn:Sem} gives the language to describe generalized modal algebras and generalized modal spaces, as well as its semantics. Section \ref{Sec:Prelim:Alg:Cor} gives the expanded language that the algorithm $\mathsf{ALBA}$ manipulate as well as the first-order correspondence language. Section \ref{sec:Sahlqvist} defines the inductive inequalities for our language. Section \ref{Sec:ALBA} describes the algorithm $\mathsf{ALBA}$ that computes the first-order correspondents of the modal formulas. Section \ref{Sec:Success} shows that $\mathsf{ALBA}$ succeeds on inductive inequalities. Section \ref{Sec:Soundness} proves that $\mathsf{ALBA}$ is sound with respect to admissible valuations. Section \ref{Sec:Example} gives an example how $\mathsf{ALBA}$ executes.

\section{Preliminaries}\label{Sec:Prelim}
In the present section, we give preliminaries on generalized Boolean algebras, generalized Stone spaces, generalized modal algebras, generalized modal spaces and their object-level duality. For more details, see \cite{St35}.

\subsection{Generalized Boolean Algebras}\label{Subsec:Tar:Alg}
\begin{definition}
A \emph{generalized Boolean algebra} is a tuple $A=(A,\to,\land,\lor,\top)$ such that $(A,\to,\land,\lor,\top)$ is a Brouwerian algebra, and $\lor$ satisfies $a\lor b=(a\to b)\to b$.
\end{definition}
It is clear that a generalized Boolean algebra satisfies all the classical propositional logic laws that do not involve $\bot$ or $\neg$. In a generalized Boolean algebra we can define an order relation $\leq$ by $a\leq b$ iff $a=a\land b$. A non-empty subset $F$ of $A$ is a called an \emph{filter} if $\top\in F$ and $F$ is upward-closed (i.e.\ if $a\in F$ and $a\leq b$ then $b\in F$) and for any $a,b\in A$, if $a,b\in F$, then $a\land b\in F$. A proper filter $F\subsetneq A$ is \emph{maximal} if for any filter $H$ such that $F\subseteq H$, we have $F=H$ or $H=A$. The set of all maximal filters of $A$ is denoted by $Ul(A)$. Given a full set $X$, we use $Y^{c}$ to denote the complement of $Y$ relative to $X$, i.e.\ $X-Y$.

\subsection{Generalized Stone Spaces}

In what follows, a topological space $(X,\tau_{\mathcal{K}})$ with basis $\mathcal{K}$ will be denoted by $(X,\mathcal{K})$.

\begin{definition}\label{Def:Tarski:Space}
A \emph{generalized Stone space} is a topological space $(X,\mathcal{K})$ such that:
\begin{itemize}
\item For every $x,y\in X$, if $x\neq y$, then there exists $U\in\mathcal{K}$ such that $x\in U$ and $y\notin U$;
\item $\mathcal{K}$ is a basis of compact subsets for the topology $\tau_{\mathcal{K}}$;
\item For any $A,B\in\mathcal{K}$, $A\cap B^{c}\in \mathcal{K}$;
\item For any $A,B\in\mathcal{K}$, $A\cup B\in \mathcal{K}$;
\item For any closed subset $Y$ of $X$, and for any downward-directed subfamily $\mathcal{I}$ of $\mathcal{K}$, if $Y\cap U\neq\emptyset$ for each $U\in I$, then $Y\cap\bigcap\mathcal{I}\neq\emptyset$.
\end{itemize}

Given a generalized Stone space $(X,\mathcal{K})$, the family $\mathcal{D}_{\mathcal{K}}(X)=\{U\subseteq X:U^{c}\in K\}$ with the implication defined by $U\Rightarrow V=U^{c}\cup V$, conjunction defined by intersection, disjunction defined by union, for each $U,V\in\mathcal{D}_{\mathcal{K}}(X)$, is a generalized Boolean algebra called the \emph{dual generalized Boolean algebra} of $(X,\mathcal{K})$.

Condition 5. of Definition \ref{Def:Tarski:Space} can be expressed in the following way (see \cite{Ce22}):

\begin{itemize}
\item For any closed subset $Y$ of $X$, and for any directed subfamily $\mathcal{F}$ of $\mathcal{D}_{\mathcal{K}}(X)$, if $Y\subseteq \bigcup\{U:U\in\mathcal{F}\}$, then $Y\subseteq U$, for some $U\in\mathcal{F}$.
\end{itemize}
\end{definition}

\subsection{Topological Duality between Generalized Boolean Algebras and Generalized Stone Spaces}

If $A$ is a generalized Boolean algebra, then the family $\mathcal{K}_{A}=\{(\sigma(a))^{c}:a\in A\}$ is a basis for a topology $\tau_{\mathcal{K}_{A}}$ defined on $Ul(A)$, where $\sigma(a)=\{P\in Ul(A):a\in P\}$. Generalizing the results given in \cite{CeCa08} we have that the pair $(Ul(A),\mathcal{K}_{A})$ is a generalized Stone space and the map $\sigma:A\to \mathcal{D}_{\mathcal{K}_{A}}(Ul(A))$ is an algebra isomorphism.

Let $(X,\mathcal{K})$ be a generalized Stone space. Then the map $\epsilon:X\to Ul(\mathcal{D}_{\mathcal{K}}(X))$ given by $\epsilon(x)=\{U\in\mathcal{D}_{K}(X): x\in U\}$ is a homeomorphism between the topological spaces $(X,\mathcal{K})$ and $(Ul(\mathcal{D}_{\mathcal{K}}(X)),\mathcal{K}_{\mathcal{D}_{\mathcal{K}}(X)})$.

\subsection{Adding Modality $\Box$}

\begin{definition}
A \emph{generalized modal algebra} is a tuple $A=(A,\to,\land,\lor,\top,\Box)$ such that $(A,\to,\land,\lor,\top)$ is a generalized Boolean algebra, $\Box \top=\top$ and $\Box (a\land b)=\Box a\land \Box b$ for all $a,b\in A$.
\end{definition}

\begin{definition}
A \emph{generalized modal space} is a structure $(X,\mathcal{K},R)$ such that $(X,\mathcal{K})$ is a generalized Stone space, and
\begin{itemize}
\item For each $x\in X$, $R(x)$ is closed;
\item $\Box U=\{x\in X:R(x)\subseteq U\}\in \mathcal{D}_{\mathcal{K}}(X)$, for each $U\in\mathcal{D}_{\mathcal{K}}(X)$.
\end{itemize}
\end{definition}

The duality between generalized Boolean algebras and generalized Stone spaces can be naturally extended to generalized modal algebras and generalized modal spaces.

\subsection{Some Useful Propositions}

The following propositions will be useful in the proof of the topological Ackermann lemma in Section \ref{SubSubSec:Top:Ack:Lemma}.

\begin{lemma}[Generalizing Lemma 2.7 in \cite{Ce22}]\label{Lemma:Pseudo:Compactness}
Let $(X,\mathcal{K})$ be a generalized Stone space. Let $\{U_{i}:i\in I\}$ and $\{V_{j}:j\in J\}$ be non-empty
families of $\mathcal{D}_{\mathcal{K}}(X)$ such that $\bigcap\{U_{i}:i\in I\}\subseteq\bigcup\{V_{j}:j\in J\}$, then there exist $U_1,\ldots,U_n$ and $V_1,\ldots,V_k$ such that $U_1\cap\ldots\cap U_n\subseteq V_1\cup\ldots\cup V_k$.
\end{lemma}

\begin{proposition}[Generalizing Proposition 4.4 in \cite{Ce22}]\label{Prop:RY:Closed}

Let $(X,\mathcal{K},R)$ be a generalized modal space. Then $R[Y]$ is a closed subset for each closed subset $Y$ of $X$.
\end{proposition}

\section{Syntax and Semantics}\label{Sec:Syn:Sem}

In the present section, we give the syntax and semantics of the logic formulas for generalized modal algebras and generalized modal spaces. We follow the presentation of \cite{Zh21c}.

\subsection{Language and Syntax}\label{Subsec:Lan:Syn}

\begin{definition}
Given a countable set $\mathsf{Prop}$ of propositional variables, the generalized modal language $\mathcal{L}$ is defined as follows:
$$\varphi::=p \mid \top \mid\varphi\to\varphi\mid\varphi\land\varphi\mid\varphi\lor\varphi\mid \Box\varphi,$$
where $p\in \mathsf{Prop}$. We use the notation $\overline{p}$ to denote a list of propositional variables and use $\phi(\overline{p})$ to indicate that the propositional variables occur in $\phi$ are all in $\vec p$. We call a formula \emph{pure} if it does not contain propositional variables. We use the notation $\overline{\theta}$ to indicate a finite list of formulas. We use the notation $\theta(\eta/p)$ to indicate uniformly substituting $p$ by $\eta$. 

We will find it convenient to use \emph{inequalities} of the form $\phi\leq\psi$ and \emph{quasi-inequalities} of the form $\phi_1\leq\psi_1\ \&\ \ldots\ \&\ \phi_n\leq\psi_n\  \Rightarrow\ \phi\leq\psi$ in the algorithm. Intuitively, $\phi\leq\psi$ expresses the model-level truth of the implicative formula $\phi\to\psi$.
\end{definition}

\subsection{Semantics}\label{Subsec:Seman}

We interpret formulas on the dual generalized modal spaces, with two kinds of valuations, namely \emph{admissible valuations} which interpret propositional variables as elements in $\mathcal{D}_{\mathcal{K}}(X)$ (i.e.\ interpret them as elements of the dual generalized modal algebras), and \emph{arbitrary valuations} which interpret propositional variables as arbitrary subsets of the space. Notice that the admissible valuations are not necessarily clopen in the topology (they are only closed and satisfy additional conditions), in contrast to existing settings in algorithmic correspondence theory \cite{CoPa12,CoPa19} where Stone/Priestley-like dualities are used. 

\begin{definition}

In a generalized modal space $(X,\mathcal{K},R)$, (we abuse notation to use $X$ to denote the space), we call $X$ the \emph{domain} of the space.

\begin{itemize}
\item An \emph{admissible model} is a pair $M=(X,V)$ where $V:\mathsf{Prop}\to\mathcal{D}_{\mathcal{K}}(X)$ is an \emph{admissible valuation} on $X$. 
\item An \emph{arbitrary model} is a pair $M=(X,V)$ where $V:\mathsf{Prop}\to P(X)$ is an \emph{arbitrary valuation} on $X$ such that for all propositional variables $p$, $V(p)$ is an arbitrary subset of $X$. 
\end{itemize}

Given a valuation $V$, a propositional variable $p\in\mathsf{Prop}$, a subset $A\subseteq X$, we can define $V^{p}_{A}$, the \emph{$p$-variant of $V$} as follows: $V^{p}_{A}(q)=V(q)$ for all $q\neq p$ and $V^{p}_{A}(p)=A$.

Now the satisfaction relation can be defined as follows: given any generalized modal space $(X,\mathcal{K},R)$, any valuation $V$ on $X$, any $w\in X$, 

\begin{center}
\begin{tabular}{l c l}
$X,V,w\Vdash p$ & iff & $w\in V(p)$\\
$X,V,w\Vdash \top$ & : & always\\
$X,V,w\Vdash\varphi\to\psi$ & iff & $X,V,w\nVdash \varphi$ or $X,V,w\Vdash\psi$\\
$X,V,w\Vdash\varphi\land\psi$ & iff & $X,V,w\Vdash \varphi$ and $X,V,w\Vdash\psi$\\
$X,V,w\Vdash\varphi\lor\psi$ & iff & $X,V,w\Vdash \varphi$ or $X,V,w\Vdash\psi$\\
$X,V,w\Vdash \Box\varphi$ & iff & $\forall v(Rwv\ \Rightarrow\ X,V,v\Vdash\varphi)$
\label{page:downarrow}\\
\end{tabular}
\end{center}
For any formula $\phi$, we let $V(\phi)=\{w\in X\mid X,V,w\Vdash\varphi\}$ denote the \emph{truth set} of $\varphi$ in $(X,V)$. 
\begin{itemize}
\item The formula $\varphi$ is \emph{globally true} on $(X,V)$ (notation: $X,V\Vdash\varphi$) if $X,V,w\Vdash\varphi$ for every $w\in W$. 
\item We say that $\varphi$ is \emph{admissibly valid} on a generalized modal space $X$ (notation: $X\Vdash_{\mathcal{D}_{\mathcal{K}}}\varphi$) if $\varphi$ is globally true on $(X,V)$ for every admissible valuation $V$.
\item We say that $\varphi$ is \emph{valid} on a generalized modal space $X$ (notation: $X\Vdash_{P}\varphi$) if $\varphi$ is globally true on $(X,V)$ for every arbitrary valuation $V$.
\end{itemize}
\end{definition}
For the semantics of inequalities and quasi-inequalities,
\begin{itemize}
\item $X,V\Vdash\phi\leq\psi\mbox{ iff }V(\phi)\subseteq V(\psi);$

\item $X,V\Vdash\phi_1\leq\psi_1\ \&\ \ldots \ \&\ \phi_n\leq\psi_n\ \Rightarrow\ \phi\leq\psi\mbox{ iff }$

$X,V\Vdash\phi\leq\psi\mbox{ holds whenever }X,V\Vdash\phi_i\leq_i\psi_i\mbox{ for all }i=1,\ldots, n.$
\end{itemize}
The definitions of validity are similar to formulas.

\section{Preliminaries on Algorithmic Correspondence}\label{Sec:Prelim:Alg:Cor}
In this section, we give preliminaries on the correspondence algorithm $\mathsf{ALBA}$ in the style of \cite{CoPa12,Zh21c}. The algorithm $\mathsf{ALBA}$ transforms the input inequality $\phi\leq\psi$ into an equivalent pure quasi-inequality which contains no propositional variable, and therefore can be translated into the first-order correspondence language via the standard translation of the expanded language (see page \pageref{Sub:FOL:ST}).

In the remainder of the paper, we will define an expanded language which the algorithm will manipulate (Section \ref{Sub:expanded:language}), define the first-order correspondence language of the expanded language and the standard translation (Section \ref{Sub:FOL:ST}). We give the definition of inductive inequalities (Section \ref{sec:Sahlqvist}), define a version of the algorithm $\mathsf{ALBA}$ (Section \ref{Sec:ALBA}), and show its success on inductive inequalities (Section \ref{Sec:Success}) and soundness with respect to admissible valuations (Section \ref{Sec:Soundness}).

\subsection{The Expanded Language}\label{Sub:expanded:language}

In the present subsection, we give the definition of the expanded language, which will be used in $\mathsf{ALBA}$:
$$\varphi::=p \mid \nomi \mid\cnomm\mid \bot \mid \top \mid \varphi\land\varphi \mid \varphi\lor\varphi \mid \varphi\to\varphi \mid \Box\varphi \mid \Diamondblack\phi$$

where $\nomi\in\mathsf{Nom}$ is called a \emph{nominal}, and $\cnomm\in\mathsf{CoNom}$ is called a \emph{conominal}. For $\nomi$, it is interpreted as a singleton set, and $\cnomm$ is interpreted as the complement of a singleton set. For $\Diamondblack$, it is interpreted as the diamond modality on the inverse relation $R^{-1}$.

For the semantics of the expanded language, the valuation $V$ is extended to $\mathsf{Prop}\cup\mathsf{Nom}\cup\mathsf{CoNom}$ such that $V(\nomi)$ is a singleton for each $\nomi\in\mathsf{Nom}$ and $V(\cnomm)$ is the complement of a singleton for each $\cnomm\in\mathsf{CoNom}$. \footnote{Notice that we allow admissible valuations to interpret nominals as singletons, even if singletons might not be in $\mathcal{D}_{\mathcal{K}}(X)$. The admissibility restrictions are only for the propositional variables.} The additional semantic clauses can be given as follows:
\begin{center}
\begin{tabular}{l c l}
$X,V,w\Vdash \bot$ & : & never\\
$X,V,w\Vdash\nomi$ & iff & $V(\nomi)=\{w\}$\\
$X,V,w\nVdash\cnomm$ & iff & $V(\cnomm)=\{w\}^{c}$\\
$X,V,w\Vdash\Diamondblack\varphi$ & iff & $\exists v(Rvw\ \mbox{ and }\ X,V,v\Vdash\varphi)$\\
\end{tabular}
\end{center}

\subsection{The First-order Correspondence Language and the Standard Translation}\label{Sub:FOL:ST}

In the first-order correspondence language, we have a binary predicate symbol $R$ corresponding to the binary relation in the generalized modal space, a set of unary predicate symbols $P$ corresponding to each propositional variable $p$.

\begin{definition}
The standard translation of the expanded language is defined as follows:
\begin{center}
\begin{tabular}{l l}
$ST_{x}(p):=Px$ & $ST_{x}(\phi\land\psi):=ST_{x}(\phi)\land ST_{x}(\psi)$\\
$ST_{x}(\bot):=\bot$ & $ST_{x}(\phi\lor\psi):=ST_{x}(\phi)\lor ST_{x}(\psi)$\\
$ST_{x}(\top):=\top$ & $ST_{x}(\phi\to\psi):=ST_{x}(\phi)\to ST_{x}(\psi)$\\
$ST_{x}(\nomi):=x=i$ & $ST_{x}(\Box\phi):=\forall y(Rxy\to ST_{y}(\phi))$\\
$ST_{x}(\cnomm):=x\neq m$ & $ST_{x}(\Diamondblack\phi):=\exists y(Ryx\land ST_{y}(\phi))$\\
\end{tabular}
\end{center}
\begin{center}
$ST(\phi\leq\psi):=\forall x(ST_{x}(\phi)\to ST_{x}(\psi))$

$ST(\phi_1\leq\psi_1\ \&\ \ldots \ \&\ \phi_n\leq\psi_n\ \Rightarrow\ \phi\leq\psi):=ST(\phi_1\leq\psi_1)\land\ldots\land ST(\phi_n\leq\psi_n)\to ST(\phi\leq\psi)$
\end{center}
\end{definition}

It is easy to see that this translation is correct:

\begin{proposition}
For any generalized modal space $X$, any valuation $V$ on $X$, any $w\in X$ and any expanded language formula $\phi$, 
$$X,V,w\Vdash\phi\mbox{ iff }X,V\vDash ST_{x}(\phi)[w].$$
\end{proposition}

\begin{proposition}\label{Prop:ST:ineq:quasi:mega}
For any generalized modal space $X$, any valuation $V$ on $X$, and inequality $\mathsf{Ineq}$, quasi-inequality $\mathsf{Quasi}$,
\begin{center}
$X,V\Vdash\mathsf{Ineq}\mbox{ iff }X,V\vDash ST(\mathsf{Ineq});$

$X,V\Vdash\mathsf{Quasi}\mbox{ iff }X,V\vDash ST(\mathsf{Quasi}).$
\end{center}
\end{proposition}

\section{Inductive Inequalities for Generalized Modal Algebras}\label{sec:Sahlqvist}

In this section, we define inductive inequalities for generalized modal algebras and generalized modal spaces.

We first define positive formulas with propositional variables in $A\subseteq\mathsf{Prop}$ as follows:
$$\mathsf{POS}_{A}::=p\mid\top\mid\Box\mathsf{POS}\mid\mathsf{POS}\land\mathsf{POS}\mid\mathsf{POS}\lor\mathsf{POS}$$
where $p\in A\subseteq\mathsf{Prop}$.

Then we define the dependence order on propositional variables as any irreflexive and transitive binary relation $<_\Omega$ on them. Then we define the PIA formulas with main variable $p$ as follows:
$$\mathsf{PIA}_{p}::=p\mid\top\mid\Box\mathsf{PIA}_{p}\mid\mathsf{POS}_{A_{p}}\to\mathsf{PIA}_{p}$$

where $A_{p}=\{q\in \mathsf{Prop}\mid q<_{\Omega}p\}$ in $\mathsf{PIA}_{p}$. Then we define the inductive antecedent as follows:
$$\mathsf{Ant}::=\mathsf{PIA}_{p}\mid\mathsf{Ant}\land\mathsf{Ant}\mid\mathsf{Ant}\lor\mathsf{Ant}$$

where $p\in\mathsf{Prop}$. Then we define the inductive succedent as follows:
$$\mathsf{Suc}::=p\mid\top\mid\mathsf{PIA}_{q}\to\mathsf{Suc}\mid\Box\mathsf{Suc}\mid\mathsf{Suc}\land\mathsf{Suc}\mid\mathsf{Suc}\lor\mathsf{Suc}$$

where $p,q\in\mathsf{Prop}$.

Finally, an \emph{$\Omega$-inductive inequality} is an inequality of the form $\mathsf{Ant}\leq\mathsf{Suc}$ where each propositional variable in the inequality both occur positively and negatively\footnote{We say that an occurrence of a propositional variable $p$ in a formula $\phi$ is \emph{positive}, if it is in the scope of an even number of the left-handside of implications, and \emph{negative} if it is in the scope of an odd number of the left-handside of implications. For example, in $(p\to q)\to r$, $p$ is positive because it is in the scope of 2 of the left-handside of implications, and $r$ is also positive (in the scope of 0), and $q$ is negative (in the scope of 1).}. An inductive inequality is an $\Omega$-inductive inequality for some $\Omega$.

\section{Algorithm}\label{Sec:ALBA}
In the present section, we define the algorithm $\mathsf{ALBA}$ which compute the first-order correspondence of the input inequality in the style of \cite{CoPa12}. The algorithm $\mathsf{ALBA}$ proceeds in three stages. Firstly, $\mathsf{ALBA}$ receives an inequality $\mathsf{Ant}\leq\mathsf{Suc}$ as input.

\begin{enumerate}

\item \textbf{Preprocessing and First approximation}:

\begin{enumerate}
\item We apply the following distribution rules exhaustively:

\begin{itemize}
\item In $\mathsf{Ant}$, rewrite every subformula of the former form into the latter form:

\begin{itemize}
\item $\alpha\land(\beta\lor\gamma)$, $(\alpha\land\beta)\lor(\alpha\land\gamma)$
\item $(\beta\lor\gamma)\land\alpha$, $(\beta\land\alpha)\lor(\gamma\land\alpha)$
\end{itemize}
\item In $\mathsf{Suc}$, rewrite every subformula of the former form into the latter form:
\begin{itemize}
\item $\alpha\to\beta\land\gamma$, $(\alpha\to\beta)\land(\alpha\to\gamma)$
\item $\Box(\alpha\land\beta)$, $\Box\alpha\land\Box\beta$
\item $\alpha\lor(\beta\land\gamma)$, $(\alpha\lor\beta)\land(\alpha\lor\gamma)$
\item $(\beta\land\gamma)\lor\alpha$, $(\beta\lor\alpha)\land(\gamma\lor\alpha)$
\end{itemize}
\end{itemize}

\item Apply the splitting rules:

$$\infer{\alpha\leq\beta\ \ \ \alpha\leq\gamma}{\alpha\leq\beta\land\gamma}
\qquad
\infer{\alpha\leq\gamma\ \ \ \beta\leq\gamma}{\alpha\lor\beta\leq\gamma}
$$

\end{enumerate}

Now for each obtained inequality $\phi_i\leq\psi_i$, We apply the following first-approximation rule:
$$\infer{\nomi_0\leq\phi_i\ \&\ \psi_i\leq\cnomm_0 \Rightarrow\ \nomi_0\leq\cnomm_0}{\phi_i\leq\psi_i}$$
Now we focus on each set of the antecedent inequalities $\{\nomi_0\leq\phi_i, \psi_i\leq\cnomm_0\}$, which we call a \emph{system}.

\item \textbf{The reduction-elimination cycle}:

In this stage, for each $\{\nomi_0\leq\phi_i,\psi_i\leq\cnomm_0\}$, we apply the following rules to eliminate all the propositional variables:
\begin{enumerate}
\item Splitting rules:

$$
\infer{\alpha\leq\beta\ \ \ \alpha\leq\gamma}{\alpha\leq\beta\land\gamma}
\qquad
\infer{\alpha\leq\gamma\ \ \ \beta\leq\gamma}{\alpha\lor\beta\leq\gamma}
$$

\item Residuation rules:\label{Page:Residuation:Rules}
\begin{prooftree}
\AxiomC{$\alpha\leq\Box\beta$}
\UnaryInfC{$\Diamondblack\alpha\leq\beta$}
\AxiomC{$\alpha\leq\beta\to\gamma$}
\UnaryInfC{$\alpha\land\beta\leq\gamma$}
\noLine\BinaryInfC{}
\end{prooftree}

\item Approximation rules:
$$\infer{\alpha\leq\cnomn\ \ \ \Box\cnomn\leq\cnomm}{\Box\alpha\leq\cnomm}
\qquad
\infer{\nomj\leq\alpha\ \ \ \ \ \ \ \beta\leq\cnomn\ \ \ \ \ \ \ \nomj\rightarrow\cnomn\leq\cnomm}{\alpha\rightarrow\beta\leq\cnomm}
$$
The nominals and conominals introduced by the approximation rules must not occur in the system before applying the rule.

\item Deleting rule: delete inequalities of the form $\alpha\leq\top$.
\item The right-handed Ackermann rule.\footnote{Here we only have the right-handed Ackermann rule, because we could only guarantee the right-handed topological Ackermann lemma, due to the fact that admissible valuations are not clopen anymore, but only closed.} This rule eliminates propositional variables and is the core of the algorithm, the other rules are aimed at reaching a shape in which the rule can be applied. Notice that an important feature of this rule is that it is executed on the whole set of inequalities, and not on a single inequality.\\

The system 
$\left\{ \begin{array}{ll}
\theta_1\leq p \\
\vdots\\
\theta_n\leq p \\
\eta_1\leq\iota_1\\
\vdots\\
\eta_m\leq\iota_m\\
\end{array} \right.$ 
is replaced by 
$\left\{ \begin{array}{ll}
\eta_1((\theta_1\lor\ldots\lor\theta_n)/p)\leq\iota_1((\theta_1\lor\ldots\lor\theta_n)/p) \\
\vdots\\
\eta_m((\theta_1\lor\ldots\lor\theta_n)/p)\leq\iota_m((\theta_1\lor\ldots\lor\theta_n)/p) \\
\end{array} \right.$
where:
\begin{enumerate}
\item $p$ does not occur in $\theta_1, \ldots, \theta_n$;
\item Each $\eta_i$ is positive, and each $\iota_i$ negative in $p$, for $1\leq i\leq m$.
\end{enumerate}
\end{enumerate}

\item \textbf{Output}: If in the previous stage, for some systems, the algorithm gets stuck, i.e.\ some propositional variables cannot be eliminated, then the algorithm halts and output ``failure''. Otherwise, each initial system after the first approximation has been reduced to a set of pure inequalities Reduce$(\nomi_0\leq\phi_i, \psi_i\leq\cnomm_0)$, and then the output is a set of quasi-inequalities $\{\&$Reduce$(\nomi_0\leq\phi_i, \psi_i\leq\cnomm_0)\Rightarrow \nomi_0\leq \cnomm_0\}_{i\in I}$. Then we can use the conjunction of the standard translations of the quasi-inequalities to obtain the first-order correspondence (notice that in the standard translation of each quasi-inequality, we need to universally quantify over all the individual variables).
\end{enumerate}

\section{Success of $\mathsf{ALBA}$}\label{Sec:Success}

In the present section, we show the success of $\mathsf{ALBA}$ on any inductive inequality $\phi\leq\psi$. 

\begin{theorem}
$\mathsf{ALBA}$ succeeds on any inductive inequality $\phi\leq\psi$ and outputs a pure quasi-inequality and a first-order formula.
\end{theorem}

\begin{proof}
We check the shape of the inequality or system in each stage, for the input inequality $\mathsf{Ant}\leq\mathsf{Suc}$:

\textbf{Stage 1.} 

After applying the distribution rules, it is easy to see that $\mathsf{Ant}$ becomes the form $\bigvee\bigwedge\mathsf{PIA}_{p}$, and $\mathsf{Suc}$ becomes the form $\bigwedge\mathsf{Suc}'$, where $$\mathsf{Suc}'::=p\mid\top\mid\mathsf{PIA}_{q}\to\mathsf{Suc}'\mid\Box\mathsf{Suc}'\mid\mathsf{Suc}'\lor\mathsf{Suc}'.$$

Then by applying the splitting rules, we get a set of inequalities of the form $\bigwedge\mathsf{PIA}_{p}\leq\mathsf{Suc}'$.

After the first approximation rule, each system is of the form $\{\nomi_0\leq\bigwedge\mathsf{PIA}_{p}, \mathsf{Suc}'\leq\cnomm_0\}$.

\textbf{Stage 2.}
In this stage, we deal with each system $\{\nomi_0\leq\bigwedge\mathsf{PIA}_{p}, \mathsf{Suc}'\leq\cnomm_0\}$.

For the inequality $\nomi_0\leq\bigwedge\mathsf{PIA}_{p}$, by first applying the splitting rule for $\land$ and then exhaustively applying the residuation rules for $\Box$ and $\to$, we get inequalities of the form $\mathsf{MinVal}_{p}\leq p$ or $\mathsf{MinVal}_{p}\leq \top$, where 
$$\mathsf{MinVal}_{p}::=\nomi_0\mid\Diamondblack\mathsf{MinVal}_{p}\mid \mathsf{MinVal}_{p}\land\mathsf{POS}_{A_{p}},$$
where $A_{p}=\{q\in \mathsf{Prop}\mid q<_{\Omega}p\}$.

For the inequality $\mathsf{Suc}'\leq\cnomm_0$, by exhaustively applying the approximation rules for $\Box$ and $\to$ and the splitting rule for $\lor$, we get a set of inequalities, each of which is in the following form:

\begin{itemize}
\item $\nomj\leq\mathsf{PIA}'_{p}$;
\item $p\leq\cnomm$;
\item $\nomi\to\cnomn\leq\cnomm$;
\item $\Box\cnomn\leq\cnomm$;
\end{itemize}

Now by applying the same strategy of the $\nomi_0\leq\bigwedge\mathsf{PIA}_{p}$ case to $\nomj\leq\mathsf{PIA}'_{p}$, we get a set of inequalities of the following form (it is easy to see that by the positivity and negativity requirement of propositional variables in inductive inequalities, the first kind of inequality always exists for each propositional variable in the inequality):

\begin{itemize}
\item $\mathsf{MinVal}_{p}\leq p$;
\item $\mathsf{MinVal}_{p}\leq \top$ (they are deleted by the deleting rule);
\item $p\leq\cnomm$;
\item $\nomi\to\cnomn\leq\cnomm$;
\item $\Box\cnomn\leq\cnomm$;
\end{itemize}

Now we are ready to apply the right-handed Ackermann rule to an $\Omega$-miminal variable $q$ to eliminate it. Then since there are only finitely many propositional variables, we can always find another $\Omega$-miminal variable to eliminate. Finally we eliminate all propositional variables and get a pure quasi-inequality and its standard translation. Notice that during last phase of the algorithm applying the right-handed Ackermann rule, the right-hand side of an inequality is either $p$ or $\cnomm$ where $p$ is only appearing in the inequalities with minimal valuations.
\end{proof}

\section{Soundness of $\mathsf{ALBA}$}\label{Sec:Soundness}

In this section we show the soundness of the algorithm with respect to the admissible valuations. The soundness proof follows the style of \cite{CoPa12}. For most of the rules, the soundness proofs are the same to existing literature and hence are omitted, so we only give details for the proofs which are different, namely the right-handed Ackermann rule.

\begin{theorem}[Soundness]\label{Thm:Soundness}
If $\mathsf{ALBA}$ runs according to the success proof on an input inductive inequality $\phi\leq\psi$ and outputs a first-order formula $\mathsf{FO(\phi\leq\psi)}$, then for any generalized modal space $(X,\mathcal{K},R)$, $$X\Vdash_{\mathcal{D}_{\mathcal{K}}}\mathsf{\phi\leq\psi}\mbox{ iff }X\vDash\mathsf{FO(\phi\leq\psi)}.$$
\end{theorem}

\begin{proof}
The proof goes similarly to \cite[Theorem 8.1]{CoPa12}. Let $\phi\leq\psi$ denote the input inequality, let $\{\nomi_{0}\leq\phi_i\ \&\ \psi_i\leq\cnomm_{0}\Rightarrow \nomi_{0}\leq\cnomm_{0}\}_{i\in I}$ denote the set of quasi-inequalities after the first-approximation rule, let $\{\&\mbox{Reduce}(\nomi_{0}\leq\phi, \psi\leq\cnomm_{0})\Rightarrow \nomi_{0}\leq\cnomm_{0}\}_{i\in I}$ denote the set of quasi-inequalities after Stage 2, let $\mathsf{FO(\phi\leq\psi)}$ denote the standard translation of the quasi-inequalities in Stage 3 into first-order formulas, then it suffices to show the equivalence from (\ref{Crct:Eqn0}) to (\ref{Crct:Eqn4}) given below:
\begin{eqnarray}
&&X\Vdash_{\mathcal{D}_{\mathcal{K}}}\phi\leq\psi\label{Crct:Eqn0}\\
&&X\Vdash_{\mathcal{D}_{\mathcal{K}}}\nomi_{0}\leq\phi_i\ \&\ \psi_i\leq\cnomm_{0}\Rightarrow \nomi_{0}\leq\cnomm_{0},\mbox{ for all }i\in I\label{Crct:Eqn2}\\
&&X\Vdash_{\mathcal{D}_{\mathcal{K}}}\&\mbox{Reduce}(\nomi_{0}\leq\phi, \psi\leq\cnomm_{0})\Rightarrow \nomi_{0}\leq\cnomm_{0},\mbox{ for all }i\in I\label{Crct:Eqn3}\\
&&X\vDash\mathsf{FO(\phi\leq\psi)}\label{Crct:Eqn4}
\end{eqnarray}
The equivalence between (\ref{Crct:Eqn0}) and (\ref{Crct:Eqn2}) follows from Proposition \ref{prop:Soundness:first:approximation};

The equivalence between (\ref{Crct:Eqn2}) and (\ref{Crct:Eqn3}) follows from Propositions \ref{Prop:Stage:2}, \ref{Prop:Ackermann};

The equivalence between (\ref{Crct:Eqn3}) and (\ref{Crct:Eqn4}) follows from Proposition \ref{Prop:ST:ineq:quasi:mega}.
\end{proof}

In the remainder of this section, we prove the soundness of the rules in each stage.
\begin{proposition}\label{prop:Soundness:first:approximation}
The distribution rules, the splitting rules are sound in $X$, and the first-approximation rule is sound in $X$, i.e.\ (\ref{Crct:Eqn0}) and (\ref{Crct:Eqn2}) are equivalent.
\end{proposition}

\begin{proof}
See \cite[Proposition 6.2, 6.3]{Zh21c}.
\end{proof}

The next step is to show the soundness of each rule of Stage 2. For each rule, before the application of this rule we have a system $S$, after applying the rule we get a system $S'$, the soundness of Stage 2 is then the equivalence of the following:
\begin{itemize}
\item $X\Vdash_{\mathcal{D}_{\mathcal{K}}}\bigamp S\ \Rightarrow \nomi_0\leq\cnomm_0$
\item $X\Vdash_{\mathcal{D}_{\mathcal{K}}}\bigamp S'\ \Rightarrow \nomi_0\leq\cnomm_0$
\end{itemize}

where $\bigamp S$ denote the meta-conjunction of inequalities of $S$. It suffices to show the following property:

\begin{itemize}\label{condition:1:4:equivalence}
\item For any $X$, any admissible valuation $V$, if $X,V\Vdash S$, then there is an admissible valuation $V'$ such that $V'(\nomi_0)=V(\nomi_0)$, $V'(\cnomm_0)=V(\cnomm_0)$ and $X,V'\Vdash S'$;
\item For any $X$, any admissible valuation $V'$, if $X,V'\Vdash S'$, then there is an admissible valuation $V$ such that $V(\nomi_0)=V'(\nomi_0)$, $V(\cnomm_0)=V'(\cnomm_0)$ and $X,V\Vdash S$.
\end{itemize}

\begin{proposition}\label{Prop:Stage:2}
The splitting rules, the approximation rules, the residuation rules and the deleting rule in Stage 2 are sound in both directions in $X$.
\end{proposition}

\begin{proof}
The soundness proofs for the splitting rules, the approximation rules and the residuation rules are the same to the soundness of the same rules in \cite[Lemma 6.5-6.7, 6.12]{Zh21c}. The soundness of the deleting rule is trivial.
\end{proof}

\begin{proposition}\label{Prop:Ackermann}
The right-handed Ackermann rule applied in the success proof is sound in $X$.
\end{proposition}

This rule is the most complicated rule in the soundness proof of the algorithm, and indeed this proof is different from existing settings due to that the admissible valuations are not clopen anymore. We devote the next subsection to this proof. 

\subsection{Right-handed Topological Ackermann Lemma}

\subsubsection{Analysis of the Right-handed Ackermann Rule}
We consider the application of the right-handed Ackermann rule. Before the application, each inequality in the system is of the following shape, where $p$ is the current $\Omega$-minimal propositional variable:

\begin{itemize}
\item $\mathsf{MinVal}'_{p,1}\leq p, \ldots, \mathsf{MinVal}'_{p,n}\leq p$
\item inequalities of the form $\mathsf{MinVal}'_q\leq q$, where $q\neq p$
\item inequalities of the form $p\leq\cnomm_p$
\item inequalities of the form $q\leq\cnomm_q$, where $q\neq p$
\item pure inequalities
\end{itemize}

here $\mathsf{MinVal}'_r\in\mathsf{C}_{r}$ for $r=p,q$, where $\mathsf{C}_{r}$ is defined as follows: 

$$\mathsf{C}_{r}::=\nomi\mid\top\mid s\mid\Diamondblack\mathsf{C}_{r}\mid\Box\mathsf{C}_{r}\mid\mathsf{C}_{r}\land\mathsf{C}_{r}\mid\mathsf{C}_{r}\lor\mathsf{C}_{r}$$

where $s$ is a propositional variable of dependence order below $r$. It is easy to see that each $\mathsf{MinVal}'_{p,i}$ is pure (since all propositional variables below $p$ are already eliminated), and $\mathsf{MinVal}'_q$ may or may not contain $p$.

Now denote $\bigvee_i\mathsf{MinVal}'_{p,i}$ as $\mathsf{V}_{p}$. After the application of the right-handed Ackermann rule, the system is of the following shape:

\begin{itemize}
\item inequalities of the form $\mathsf{MinVal}'_q(\mathsf{V}_{p}/p)\leq q$
\item inequalities of the form $\mathsf{V}_{p}\leq\cnomm_p$
\item inequalities of the form $q\leq\cnomm_q$
\item pure inequalities
\end{itemize}

It is easy to see that in the system, in each non-pure inequality, they are of the form $\mathsf{MinVal}'_q(\mathsf{V}_{p}/p)\leq q$ or $q\leq\cnomm_q$, which still fall in the categories described as before the application of the right-handed Ackermann rule. Also, the first application of the right-handed Ackermann rule is a special case of the situation described above.

Therefore, it suffices to show that for the system 

\begin{itemize}
\item $\mathsf{MinVal}'_{p,1}\leq p, \ldots, \mathsf{MinVal}'_{p,n}\leq p$
\item inequalities of the form $\mathsf{MinVal}'_q\leq q$, where $q\neq p$
\item inequalities of the form $p\leq\cnomm_p$
\item inequalities of the form $q\leq\cnomm_q$, where $q\neq p$
\item pure inequalities,
\end{itemize}

the application of the right-handed Ackermann rule on variable $p$ is sound with respect to admissible valuations. Indeed, we can consider the system as follows (without loss of generality we can take the join of the minimal valuations for $p$ to make them $\mathsf{V}_{p}$):\label{Page:Before:Ackermann}

\begin{itemize}
\item $\mathsf{V}_{p}\leq p$
\item inequalities of the form $\mathsf{MinVal}'_q(p)\leq q$, where $q\neq p$ and $\mathsf{MinVal}'_q(p)$ contains positive occurrences of $p$
\item inequalities of the form $p\leq\cnomm_p$
\item inequalities that does not contain $p$.
\end{itemize}

\subsubsection{Proof of Topological Ackermann Lemma}\label{SubSubSec:Top:Ack:Lemma}
In what follows we denote $\mathcal{D}_{\mathcal{K}}(X)$ as $A$, which means that we identify the generalized Boolean algebra $\mathcal{D}_{\mathcal{K}}(X)$ on $X$ and the dual generalized Boolena algebra $A$, as well as their elements. We use $C(X)$ to denote the set of closed elements in $X$. We also denote $\Diamondblack Y:=R[Y]$ and $\Box Y:=(R^{-1}(Y^{c}))^{c}$. Now we prove the following lemmas:

\begin{lemma}
\begin{enumerate}
\item For any $Y\in C(X)$, $Y$ is an intersection of a downward-directed collection of elements in $A$, i.e.\ $\forall Y\in C(X)$, $Y=\bigcap_i X_i$ for some downward-directed $\{X_i\}_{i\in I}\subseteq A$.
\item $A\subseteq C(X)$.
\item $\{x\}\in C(X)$ for any $x\in X$.
\item If $Y\in C(X)$, $\Diamondblack Y\in C(X)$.
\item $\Box\bigcap_i X_i=\bigcap\Box_i X_i$, for any $X_i\in P(X)$. 
\item If $Y\in C(X)$, then $\Box Y\in C(X)$.
\end{enumerate}
\end{lemma}

\begin{proof}
\begin{enumerate}
\item By the fact that $\mathcal{K}$ is a basis of $X$, $C$ is an intersection of elements in $A$, i.e.\ $\forall Y\in C(X)$, $Y=\bigcap_i Y_i$ for some $\{Y_j\}_{j\in J}\subseteq A$. Now define $\{X_i\}_{i\in I}$ to be the set of all finite intersections of elements in $\{Y_j\}_{j\in J}$. Since $A$ is closed under taking finite intersections, it is easy to see that $\{X_i\}_{i\in I}$ is a subset of $A$ and in downward-directed.
\item Trivial.
\item By Definition \ref{Def:Tarski:Space}, singletons are closed.
\item By Proposition \ref{Prop:RY:Closed}.
\item By the fact that $\Box$ is completely intersection preserving.
\item An easy corollary of items 1 and 5.
\end{enumerate}
\end{proof}

\begin{lemma}\label{Lemma:Pseudo:Compactness:2}
Let $(X,\mathcal{K})$ be a generalized Stone space. Let $\{U_{i}:i\in I\}$ be non-empty downward-directed
family in $C(X)$ and $V\in A$ such that $\bigcap\{U_{i}:i\in I\}\subseteq V$, then there exists an $i\in I$ such that $U_i\subseteq V$.
\end{lemma}

\begin{proof}
For each $i\in I$, consider $U_i\in C(X)$, there is a collection $\{X_{i,j}\}_{j\in J_i}\subseteq A$ such that $U_i=\bigcap_j X_{i,j}$, therefore we have $\bigcap_{i,j}X_{i,j}\subseteq V$, by Lemma \ref{Lemma:Pseudo:Compactness}, there are $X_{i_1,j_1}, \ldots, X_{i_n,j_n}$ such that $X_{i_1,j_1}\cap\ldots\cap X_{i_n,j_n}\subseteq V$, therefore $U_{i_1}\cap\ldots\cap U_{i_n}\subseteq V$, by downward-directedness, there is an $i\in I$ such that $U_i\subseteq V$.
\end{proof}

\begin{lemma}
\begin{enumerate}
\item $\Diamondblack\bigcap_i X_i=\bigcap\Diamondblack_i X_i$ for any non-empty downward-directed $\{X_i\}_{i\in I}\subseteq A$.
\item $\Diamondblack\bigcap_i X_i=\bigcap\Diamondblack_i X_i$ for any non-empty downward-directed $\{X_i\}_{i\in I}\subseteq C(X)$.
\end{enumerate}
\end{lemma}

\begin{proof}
\begin{enumerate}
\item The direction $\Diamondblack\bigcap_i X_i\subseteq\bigcap\Diamondblack_i X_i$ is easy. For the other direction, suppose there is a closed set $Y\subseteq X$ such that $\Diamondblack\bigcap_i X_i\subseteq Y$. Then there is a collection $\{Z_j\}_{j\in J}\subseteq A$ such that $Y=\bigcap_j Z_j$. Therefore, $\Diamondblack\bigcap_i X_i\subseteq Z_j$ for all $j$. Thus $\bigcap_i X_i\subseteq \Box Z_j$ for all $j$. By Lemma \ref{Lemma:Pseudo:Compactness} and the downward-directedness of $X_i$, we have that for all $j$ and some $k_{j}$ depending on $j$, $X_{k_{j}}\subseteq \Box Z_j,\mbox{ i.e., }\Diamondblack X_{k_{j}}\subseteq Z_j,$ so $\bigcap_i\Diamondblack X_{i}\subseteq Z_j\mbox{ for all }j,$ so $\bigcap_i\Diamondblack X_{i}\subseteq \bigcap_j Z_j=Y.$ Now take $Y=\Diamondblack\bigcap_i X_i$, we have $\bigcap_i\Diamondblack X_{i}\subseteq\Diamondblack\bigcap_i X_i.$
\item The direction $\Diamondblack\bigcap_i X_i\subseteq\bigcap\Diamondblack_i X_i$ is easy. For the other direction, suppose there is a closed set $Y\subseteq X$ such that $\Diamondblack\bigcap_i X_i\subseteq Y$. Then there is a collection $\{Z_j\}_{j\in J}\in A$ such that $Y=\bigcap_j Z_j$. Therefore, $\Diamondblack\bigcap_i X_i\subseteq Z_j$ for all $j$. Thus $\bigcap_i X_i\subseteq \Box Z_j$ for all $j$. By Lemma \ref{Lemma:Pseudo:Compactness:2} and the downward-directedness of $X_i$, we have that for all $j$ and some $k_j$ depending on $j$, $X_{k_j}\subseteq \Box Z_j,\mbox{ i.e., }\Diamondblack X_{k_j}\subseteq Z_j,$ so $\bigcap_i\Diamondblack X_{i}\subseteq Z_j\mbox{ for all }j,$ so $\bigcap_i\Diamondblack X_{i}\subseteq \bigcap_j Z_j=Y.$ Now take $Y=\Diamondblack\bigcap_i X_i$, we have $\bigcap_i\Diamondblack X_{i}\subseteq\Diamondblack\bigcap_i X_i.$
\end{enumerate}
\end{proof}

\begin{lemma}\label{Lemma:Closed}
For any formula $\alpha(\overline{q}, \overline{\nomi}, p)$ built up from nominals, $\top$, propositional variables using $\Diamondblack,\Box,\land,\lor$, consider the following valuation: $V(\overline{q})=\overline{Y},\mbox{ where }\overline{Y}\in A$, $V(\overline{\nomi})=\overline{\{x\}},\mbox{ where }x\in X$, $V(p)=Z,\mbox{ where }Z\in A$, then $V(\alpha(\overline{q}, \overline{\nomi}, p))\in C(X)$.
\end{lemma}

\begin{proof}
By induction on the complexity of $\alpha(\overline{q}, \overline{\nomi}, p)$.
\end{proof}

\begin{lemma}\label{Lemma:Intersection}
For any formula $\alpha(\overline{q}, \overline{\nomi}, p)$ built up from nominals, $\top$, propositional variables using $\Diamondblack,\Box,\land,\lor$, for any $\overline{Y}\in A$ corresponding to $\overline{q}$, any $\overline{\{x\}}$ (where $x\in X$) corresponding to $\overline{\nomi}$, any downward-directed $\{Z_i\}_{i\in I}\subseteq C(X)$ corresponding to $p$, we have $\bigcap_i\beta(\overline{Y},\overline{\{x\}},Z_i)=\beta(\overline{Y},\overline{\{x\}},\bigcap_i Z_i).$\footnote{Here by $\beta(\overline{Y},\overline{\{x\}},Z_i)$ we mean $V(\beta)$ under the valuation where $V(\overline{q})=\overline{Y}$, $V(\overline{\nomi})=\overline{\{x\}}$, $V(p)=Z_i$.}
\end{lemma}

\begin{proof}
By induction on the complexity of $\alpha(\overline{q}, \overline{\nomi}, p)$.
\end{proof}

\begin{lemma}[Right-handed topological Ackermann lemma]\label{aRight:Ack}
Let $\theta$ be $\mathsf{V}_{p}$, $\eta_{i}(p)\leq\iota_{i}$ be $\mathsf{MinVal}'_{q}(p)\leq q$ or $p\leq\cnomm_p$ as described on page \pageref{Page:Before:Ackermann} for $1\leq i\leq m$. Then for any admissible valuation $V$, the following are equivalent:
\begin{enumerate}
\item
$V(\eta_i(\theta/p))\subseteq V(\iota_i)$ for $1\leq i\leq m$,
\item
there exists $Z\in A$ such that $V'(\theta)\subseteq Z$ and $V'(\eta_i(p))\subseteq V'(\iota_i)$ for $1\leq i\leq m$,  where $V'$ is the same as $V$ except that $V'(p)=Z$.
\end{enumerate}
\end{lemma}

\begin{proof}
2. to 1. is easy by monotonicity. 

For 1. to 2., since $p$ is the current $\Omega$-minimal propositional variable, $\theta$ is pure. By Lemma \ref{Lemma:Closed} and the shape of $\theta$, $V(\theta)=V'(\theta)\in C(X)$. Now we denote $V(q)$ as $Q$, then $Q\in A$, $V(\cnomm_p)=\{m\}^{c}$ for some $m\in X$.

Therefore, from $V(\mathsf{MinVal}'_{q}(\theta/p))\subseteq V(q)$, we have that $\mathsf{MinVal}'_{q}(V(\theta))\subseteq Q$\footnote{Here by $\mathsf{MinVal}'_{q}(V(\theta))$ we mean $V(\mathsf{MinVal}'_{q})$ where $V(p)=V(\theta)$.}, since $V(\theta)\in C(X)$, there is a downward-directed collection $\{X_i\}_{i\in I}\subseteq A$ such that $V(\theta)=\bigcap_i X_i$, so $\mathsf{MinVal}'_{q}(\bigcap_i X_i)\subseteq Q$. Now by Lemma \ref{Lemma:Intersection}, $\mathsf{MinVal}'_{q}(\bigcap_i X_i)=\bigcap_i\mathsf{MinVal}'_{q}(X_i)$, so $\bigcap_i\mathsf{MinVal}'_{q}(X_i)\subseteq Q$. Now by Lemma \ref{Lemma:Closed}, it is easy to see that $\{\mathsf{MinVal}'_{q}(X_i)\mid i\in I\}$ is a downward-directed set of elements in $C(X)$, so by Lemma \ref{Lemma:Pseudo:Compactness:2}, there is an $i\in I$ such that $\mathsf{MinVal}'_{q}(X_i)\subseteq Q$.

For $p\leq \cnomm_p$, we have that $V(\theta)\subseteq \{m\}^{c}$. Since $V(\theta)\in C(X)$, there are downward-directed $Y_j\in A$ ($j\in J$) such that $V(\theta)=\bigcap_j Y_j$, so $\bigcap_{j}Y_{j}\subseteq\{m\}^{c}$. So $m\notin\bigcap_{j}Y_{j}$, therefore there is a $j\in J$ such that $m\notin Y_{j}$, so $Y_{j}\subseteq \{m\}^{c}$.

Now take the intersection of all $X_i$s and $Y_j$s picked up for each $\eta_{i}(p)\leq\iota_{i}$, we get the $Z$ as desired (this is possible since $A$ is closed under taking finite intersections).
\end{proof}

Now from the right-handed topological Ackermann lemma, it is easy to see the soundness of the right-handed Ackermann rule when executed as described in the success proof.

\section{Examples}\label{Sec:Example}

Here we show the execution of an example from \cite[Proposition 6.1]{Ce22}:
\begin{example}
$\forall p(\top\leq\Box(\Box p\to p))$\\
$\forall p\forall \nomi_0\forall\cnomm_0(\nomi_0\leq\top\ \&\ \Box(\Box p\to p)\leq\cnomm_0\ \Rightarrow\ \nomi_0\leq\cnomm_0)$\\
$\forall p\forall \nomi_0\forall\cnomm_0\forall\cnomm_1(\nomi_0\leq\top\ \&\ \Box\cnomm_1\leq\cnomm_0\ \&\ \Box p\to p\leq\cnomm_1\ \Rightarrow\ \nomi_0\leq\cnomm_0)$\\
$\forall p\forall \nomi_0\forall\cnomm_0\forall\cnomm_1(\Box\cnomm_1\leq\cnomm_0\ \&\ \Box p\to p\leq\cnomm_1\ \Rightarrow\ \nomi_0\leq\cnomm_0)$\\
$\forall p\forall \nomi_0\forall \nomi_1\forall\cnomm_0\forall\cnomm_1\forall\cnomm_2(\Box\cnomm_1\leq\cnomm_0\ \&\ \nomi_1\to \cnomm_2\leq\cnomm_1\ \&\ \nomi_1\leq\Box p\ \&\ p\leq\cnomm_2\ \Rightarrow\ \nomi_0\leq\cnomm_0)$\\
$\forall p\forall \nomi_0\forall \nomi_1\forall\cnomm_0\forall\cnomm_1\forall\cnomm_2(\Box\cnomm_1\leq\cnomm_0\ \&\ \nomi_1\to \cnomm_2\leq\cnomm_1\ \&\ \Diamondblack\nomi_1\leq p\ \&\ p\leq\cnomm_2\ \Rightarrow\ \nomi_0\leq\cnomm_0)$\\
$\forall \nomi_0\forall \nomi_1\forall\cnomm_0\forall\cnomm_1\forall\cnomm_2(\Box\cnomm_1\leq\cnomm_0\ \&\ \nomi_1\to \cnomm_2\leq\cnomm_1\ \&\ \Diamondblack\nomi_1\leq\cnomm_2\ \Rightarrow\ \nomi_0\leq\cnomm_0)$\\
$\forall i_0\forall i_1\forall m_0\forall m_1\forall m_2(\Box(\{m_1\}^c)\subseteq\{m_0\}^c\land\{i_1\}^c\cup\{m_2\}^c\subseteq\{m_1\}^c\land \Diamondblack\{i_1\}\subseteq\{m_2\}^c\to \{i_0\}\subseteq\{m_0\}^c)$\\
$\forall i_0\forall i_1\forall m_0\forall m_1\forall m_2(Rm_0 m_1\land i_1=m_2=m_1\land \neg Ri_1 m_2\to i_0\neq m_0)$\\
$\forall i_0\forall i_1(Ri_0 i_1\to Ri_1 i_1)$
\end{example}

\bibliographystyle{abbrv}
\bibliography{Generalized_Boolean_Algebra}

\begin{thebibliography}{10}

\bibitem{Ab67a}
J.~C. Abbott.
\newblock Implicational algebras.
\newblock {\em Bulletin mathématique de la Société des Sciences
  Mathématiques de la République Socialiste de Roumanie}, 11 (59)(1):3--23,
  1967.

\bibitem{Ab67b}
J.~C. Abbott.
\newblock Semi-{B}oolean algebra.
\newblock {\em Matematički Vesnik}, 4(19)(40):177--198, 1967.

\bibitem{Ce05}
S.~A. Celani.
\newblock Modal {T}arski algebras.
\newblock {\em Reports on Mathematical Logic}, pages 113--126, 2005.

\bibitem{Ce19b}
S.~A. Celani.
\newblock Complete and atomic {T}arski algebras.
\newblock {\em Arch. Math. Log.}, 58(7-8):899--914, nov 2019.

\bibitem{Ce19a}
S.~A. Celani.
\newblock Subordination {T}arski algebras.
\newblock {\em Journal of Applied Non-Classical Logics}, 29(3):288--306, 2019.

\bibitem{Ce22}
S.~A. Celani.
\newblock Relational representation of quasi-modal {T}arski algebras.
\newblock 2022, in preparation.

\bibitem{CeCa08}
S.~A. Celani and L.~M. Cabrer.
\newblock Topological duality for {T}arski algebras.
\newblock {\em Algebra universalis}, 58(1):73--94, Feb 2008.

\bibitem{CoGhPa14}
W.~Conradie, S.~Ghilardi, and A.~Palmigiano.
\newblock Unified correspondence.
\newblock In A.~Baltag and S.~Smets, editors, {\em Johan van Benthem on Logic
  and Information Dynamics}, volume~5 of {\em Outstanding Contributions to
  Logic}, pages 933--975. Springer International Publishing, 2014.

\bibitem{CoPa12}
W.~Conradie and A.~Palmigiano.
\newblock Algorithmic correspondence and canonicity for distributive modal
  logic.
\newblock {\em Annals of Pure and Applied Logic}, 163(3):338 -- 376, 2012.

\bibitem{CoPa19}
W.~Conradie and A.~Palmigiano.
\newblock Algorithmic correspondence and canonicity for non-distributive
  logics.
\newblock {\em Annals of Pure and Applied Logic}, 170(9):923 -- 974, 2019.

\bibitem{St35}
M.~H. Stone.
\newblock Postulates for boolean algebras and generalized boolean algebras.
\newblock {\em American Journal of Mathematics}, 57(4):703--732, 1935.

\bibitem{Zh21c}
Z.~Zhao.
\newblock Algorithmic correspondence for hybrid logic with binder.
\newblock {\em Logic Journal of the IGPL}, 09 2021.
\newblock jzab029.

\end{thebibliography}
\end{document}